\documentclass[11pt]{amsart}
\usepackage[top=3cm, bottom=3.5cm, left=3cm, right=3cm]{geometry}
\usepackage{amsmath,amsfonts,amssymb,amsthm}
\usepackage{mathtools}
\usepackage{tikz-cd}
\usepackage[utf8]{inputenc}
\usepackage{hyperref}
\usepackage{enumerate}
\usepackage[utf8]{inputenc} 
\usepackage{mathrsfs,amsmath}
\usepackage{amsfonts}
\usepackage{tikz-cd}

\definecolor{mypink1}{rgb}{0.99, 0.82, 0.82}
\definecolor{myorange1}{rgb}{0.99, 0.8829, 0.6412}
\definecolor{myyellow1}{rgb}{0.99, 0.98, 0.6412}
\definecolor{myblue1}{rgb}{0.75, 0.889, 0.99}
\definecolor{mygreen1}{rgb}{0.77, 0.95, 0.72}
\definecolor{mypurple1}{rgb}{0.92, 0.779, 0.92}
\definecolor{myperiwinkle1}{rgb}{0.83, 0.8, 0.9912}
\definecolor{myturquoise1}{rgb}{0.74, 0.97, 0.92}
\definecolor{mygray}{gray}{0.6}
\def\ddb{\partial\bar\partial}
\usepackage{geometry} 
\usepackage{graphicx} 

\usepackage{booktabs} 
\usepackage{array} 
\usepackage{paralist} 
\usepackage{verbatim} 
\usepackage{subfig} 

\usepackage{fancyhdr} 
\author{Xi Sisi Shen}
\address{Department of Mathematics\\
  Columbia University\\
  New York, NY USA 10027}
\email[X. S. Shen]{xss@math.columbia.edu}
\author{Kevin Smith}
\address{Department of Mathematics\\
  Columbia University\\
  New York, NY USA 10027}
\email[K. Smith]{kjs@math.columbia.edu}

\newtheorem{proposition}{Proposition}

\newtheorem{thm}{Theorem}

\DeclareMathOperator\Ree{Re}
\DeclareMathOperator\Imm{Im}

\DeclareMathOperator\tr{tr}

\DeclareMathOperator\Ric{Ric}

\DeclareMathOperator\BC{BC}

\begin{document}
\bibliographystyle{amsplain}

\title{The continuity equation on Hopf and Inoue surfaces}
\maketitle

\begin{abstract}
We study the continuity equation of La Nave-Tian, extended to the Hermitian setting by Sherman-Weinkove, on Hopf and Inoue surfaces. We prove a priori estimates for solutions in both cases, and Gromov-Hausdorff convergence of Inoue surfaces to a circle.
\end{abstract}

\section{Introduction}\label{s1}

A parabolic approach to proving existence of K\"ahler-Einstein metrics on a K\"ahler manifold by means of the K\"ahler-Ricci flow was used by Cao \cite{cao85} to provide an alternative proof of the existence of K\"ahler-Einstein metrics on manifolds with $c_1(X)<0$ and $c_1(X)=0$, originally proved by Yau \cite{yau78} and also independently by and Aubin \cite{aubin78} in the case $c_1(X)<0$. The K\"ahler-Ricci flow preserves the Hermitian condition and has been shown to illuminate various complex and algebro-geometric properties of the manifold. Starting in 2007, Song-Tian \cite{song-tian07, song-tian12, st17} and Tian \cite{tian08} proposed an \textit{Analytic Minimal Model Program} to classify algebraic varieties using the K\"ahler-Ricci flow to  ``simplify" algebraic varieties to their minimal models. For example, the K\"ahler-Ricci flow can be set up to have the behavior of ``blowing down" (-1)-curves on a complex surface. La Nave-Tian propose in \cite{nt15} a continuity equation to provide an alternative method for carrying out the Analytic Minimal Model Program and show, under an assumption on the initial metric, the convergence of the metric to a weak K\"ahler-Einstein metric away from a subvariety. An advantage of the continuity equation is that solutions always have a Ricci lower bound which provides the groundwork to apply compactness theory by Cheeger-Colding-Tian and a partial $C^0$ estimate.

Beyond the K\"ahler setting, there is a classification of complex surfaces due to Enriques-Kodaira \cite{bhpv04} which states that  all minimal non-K\"ahler compact complex surfaces belong to one of the following classes:
\begin{itemize}
    \item Kodaira surfaces
    \item Minimal properly elliptic surfaces
    \item Class VII surfaces with vanishing second Betti number
    \begin{itemize}
        \item Hopf surfaces
        \item Inoue surfaces
    \end{itemize}
    \item Class VII surfaces with positive second Betti number
\end{itemize}
In particular, there is very little known about a special class of surfaces called the Class VII surfaces which are defined as complex surfaces with Kodaira dimension 1 and first Betti number 1. Hopf and Inoue surfaces are exactly those class VII surfaces with $b_2(X)=0$ \cite{bogomolov82,kodaira66,lyz94,teleman94}.

One way to extend the K\"ahler-Ricci flow to the Hermitian setting to study the complex structures on these non-K\"ahler surfaces is by the Chern-Ricci flow introduced by Gill \cite{gill11}. 
Let $(X,J,g_0)$ be a compact complex manifold of complex dimension $n$ and $g_0$ a Hermitian metric on $X$. 
The Chern-Ricci flow is given by
$$ \frac{d}{dt}\omega = -\Ric(\omega), \ \omega(0)=\omega_0,$$
where $\omega = i g_{\bar{k}j}dz^j\wedge d\bar{z}^k$ in local holomorphic coordinates, and $\Ric(\omega)$ to be the Chern-Ricci curvature associated to $g$ which is given by
\begin{align*}
    \Ric(\omega) = -i\partial\bar\partial\log\det g.
\end{align*}
For complex manifolds with $c_1^{\BC}(X)=0$, Gill \cite{gill11} proved the long time existence of the flow and smooth convergence of the flow to the unique Chern-Ricci-flat metric in the $\partial\bar\partial$-class of the initial metric. For general complex manifolds, Tosatti-Weinkove characterize the maximal existence time of the solution as well as detail the behavior of the flow in \cite{tw15}. They show that the Chern-Ricci flow can be used to contract$(-1)$ curves on a complex surface to arrive at a minimal surface. In addition, they provide a classification of complex surfaces based on the maximal existence time of the flow starting at a pluriclosed metric $\omega_0$. When $\omega_0$ is K\"ahler, the Chern-Ricci flow coincides with the K\"ahler-Ricci flow. The behavior of the Chern-Ricci flow on Hopf surfaces, Inoue surfaces and elliptic surfaces have been studied in \cite{tw13, tw15, twy15,ftwz16,edwards21, ta21}. Beyond the Chern-Ricci flow, there  exists a vast literature on complex flows of Hermitian metrics \cite{acs,bv17, bv20, bx, fppz21, fppz22, kawamura19, liu-yang12, ppz18-1, ppz18-2, shen22, streets_tian12,tw15, yury}.

In this paper, we study the continuity equation introduced by La Nave-Tian \cite {nt15} in the K\"ahler setting and extended to the Chern-Ricci case by Sherman-Weinkove\cite{sw20} in the case of Hopf and Inoue surfaces. The continuity equation starting at a Hermitian metric $\omega_0$ is given by
\begin{align}\label{s1_ContinuityEquation}
\omega = \omega_0 - s\Ric(\omega)
\end{align}
for $s\ge 0$, where $\Ric(\omega)$ is the Chern-Ricci curvature of $\omega$. This can be viewed as an elliptic alternative to the Chern-Ricci flow. One immediate advantage of this equation is a Ricci lower bound for $s\ge 0.$ The maximal existence interval for solutions to the continuity equation was shown by Sherman-Weinkove \cite{sw20}. They also prove convergence of solutions of the continuity equation on minimal properly elliptic surfaces; specifically, for elliptic bundles over a curve of genus at least 2, solutions to the continuity equation starting at a Gauduchon metric converge to the base curve in the Gromov-Hausdorff topology.

Let $\alpha = (\alpha_1,\ldots,\alpha_n)\in\mathbb{C}^n\setminus\{0\}$ with $|\alpha_1|=\cdots=|\alpha_n|\neq 1$. The round Hopf manifold is defined as $M_\alpha = (\mathbb{C}^n\setminus\{0\}) / \sim$, where $$(z^1,\ldots,z^n)\sim (\alpha_1 z^1,\ldots,\alpha_n z^n).$$ We define the metric $$\omega_{\mathrm{H}} = \frac{\delta_{\bar kj}}{r^2}dz^j\wedge d\bar z^k$$
where $r^2 = \sum_{j=1}^n |z_j|^2.$ Tosatti-Weinkove show that $\hat\omega_t := \omega_{\mathrm{H}} - t\Ric(\omega_{\mathrm{H}})$ gives an explicit solution to the Chern-Ricci flow for $t\in[0,\frac{1}{n})$ in \cite{tw13} and this solution also extends to case of the continuity equation.

\begin{thm}\label{s1_Theorem1}
Suppose that $\omega = \hat\omega + i\partial\bar\partial\varphi$ solves \eqref{s1_ContinuityEquation} on $M_\alpha\times[0,\frac{1}{n})$ with $\omega_0 = \omega_{\mathrm{H}} + i\partial\bar\partial\varphi_0$ and $\varphi$ normalized appropriately. Then
\begin{enumerate}
    \item $|\varphi|\leq C$
    \item $C^{-1}\hat\omega^n \leq \omega^n \leq C\hat\omega^n$
    \item $\omega \leq C\omega_{\mathrm{H}}$
    \item $\frac{C^{-1}}{(1-ns)^{1 - \frac{1}{n}}} \leq R(\omega)\leq \frac{C}{(1-ns)^{n-1}}$.
\end{enumerate} 
In particular, $\varphi$ converges subsequentially in $C^{1+\beta}$ as $s\to 1/n$.
\end{thm}

\noindent Using the continuity equation allows us to obtain the estimates \textit{(2)} and \textit{(4)} which are not known for solutions to the Chern-Ricci flow. The estimates \textit{(1)} and \textit{(3)} were shown for solutions to the Chern-Ricci flow by Tosatti-Weinkove in \cite{tw15}. 

We now consider a primary Hopf surface of class 1. They are defined as $M_{\alpha,\beta} = (\mathbb{C}^2\setminus \{0\})/ \sim$, where $(z,w)\sim(\alpha z,\beta w)$ for $\alpha,\beta\in\mathbb{C},$ with $1<|\alpha|\le |\beta|$. On $M_{\alpha,\beta}$, we describe a few explicit $(1,1)$-forms (see \cite{go98}, \cite{edwards21} for details). The equation
\begin{align*}
|z|^2\Phi^{-2\gamma_1} + |w|^2\Phi^{-2\gamma_2} = 1
\end{align*}
uniquely defines a smooth function $\Phi(z,w)$ on $\mathbb C^2\setminus\{0\}$, where
\begin{align*}
\gamma_1 = \frac{\log|\alpha|}{\log|\alpha| + \log|\beta|},  \ \ \gamma_2 = \frac{\log|\beta|}{\log|\alpha| + \log|\beta|}
\end{align*}

\noindent The Gauduchon-Ornea metric
\begin{align*}
\omega_{\mathrm{GO}} = \frac{i\ddb\Phi}{\Phi}
\end{align*}
is a well-defined Hermitian metric on $M_{\alpha,\beta}$. It was shown in \cite{go98}, that $\omega_{\mathrm{GO}}$ is locally conformally K\"ahler. Now define a non-negative $(1,1)$-form
\begin{align*}
\Theta = \frac{i\partial\Phi\wedge\bar\partial\Phi}{\Phi^2}.
\end{align*}
In this case, define $\hat\omega_s := (1-2s)\omega_{\mathrm{GO}}+2s\Theta$.

\begin{thm}\label{s1_Theorem2}
Suppose that $\omega = \hat\omega + i\partial\bar\partial\varphi$ solves \eqref{s1_ContinuityEquation} on $M_{\alpha,\beta}\times[0,\frac12)$ with $\omega_0 = \omega_{\mathrm{GO}} + i\partial\bar\partial\varphi_0$ and $\varphi$ normalized appropriately. Then
\begin{enumerate}
    \item $|\varphi|\leq C$
    \item $C^{-1}\hat\omega^2 \leq \omega^2 \leq C\hat\omega^2$
    \item $\omega \leq C\omega_{\mathrm{GO}}$
    \item $\frac{C^{-1}}{\sqrt{1-2s}} \leq R(\omega)\leq \frac{C}{1-2s}$.
\end{enumerate}
In particular, $\varphi$ converges subsequentually in $C^{1,\beta}$ as $s\to 1/2$.
\end{thm}

\noindent Parts \textit{(1)} and \textit{(3)} reflect  the behavior of solutions to the Chern-Ricci flow as shown by Edwards in \cite{edwards21}. However, the volume bound in \textit{(2)} and the Chern scalar curvature bound in \textit{(4)} for the continuity equation are not known for solutions to the Chern-Ricci flow.

On Inoue surfaces, we now consider the normalized continuity equation given by
\begin{align}\label{s1_NormalizedContinuityEquation}
    \omega = \omega_0 - s\Ric(\omega)-s\omega.
\end{align}
An Inoue surface $S_M$ is the quotient of $\mathbb C \times \mathbb H$ by the relations $(\mu z,\lambda w) \sim (z,w)$ and $(z + m_j, w + \ell_j) \sim (z,w)$ where $M$ is a given matrix in $SL_3(\mathbb Z)$ and $Mm = \mu m$, $M\ell = \lambda\ell$ Let $y = \Imm w$. Let $\omega'$ be a Hermitian metric which is either Gauduchon or strongly flat along the leaves in the sense of \cite{ftwz16, ta21}. We prove the following result in Section \ref{s5}:

\begin{thm}\label{s1_Theorem3}
Suppose that $\omega = \hat\omega + i\partial\bar\partial\varphi$ solves \eqref{s1_NormalizedContinuityEquation} on $S_M\times[0,\infty)$ with $\omega_0 = \omega' + i\partial\bar\partial\psi$ and $\varphi$ normalized appropriately. Then
\begin{enumerate}
    \item $|\varphi| \leq C/(s+1)$
    \item $e^{-C/s}\hat\omega^2 \leq \omega^2 \leq e^{C/s}\hat\omega^2$
    \item $(1-C(s+1)^{-1/4})\hat\omega \leq \omega \leq (1+C(s+1)^{-1/4})\hat\omega$
    \item $-C\omega \leq \Ric(\omega) \leq C\omega$
\end{enumerate}
In particular, $(S_M,\omega)$ converges to a circle $S^1$ in the Gromov-Hausdorff sense as $s\to\infty$.
\end{thm}

\noindent The above result is analogous to the behavior of the evolving metric $\omega(t)$ along the normalized Chern-Ricci flow as shown by Fang-Tosatti-Weinkove-Zheng in \cite{ftwz16}, and reflects the structure of Inoue surfaces as bundles over $S^1$. An improvement in using the continuity equation is that we are able to obtain a uniform Chern Ricci curvature bound as in \textit{(4)} without the use of higher order estimates. The results carry over straightforwardly to the Inoue surfaces of type $S^+$ and $S^-$, as in Section 3 of \cite{ftwz16}. In \cite{ta21}, Tosatti-Angella prove in fact all higher order estimates for the evolving metric along the Chern-Ricci flow as well as uniform curvature bounds on the evolving metric. 

We organize the paper as follows. In Section \ref{s2}, we cover some notation and state some general estimates that will be used in the later sections. In Section \ref{s3} we prove Theorem \ref{s1_Theorem1}, followed by Theorem \ref{s1_Theorem2} in Section \ref{s4}. Lastly, we prove Theorem \ref{s1_Theorem3} in \ref{s5}. \\

\noindent {\bf Acknowledgments:} We would like to thank Ben Weinkove, D. H. Phong, Gregory Edwards and Valentino Tosatti for useful conversations and suggestions.
\section{Preliminaries}\label{s2}

\noindent To a Hermitian metric $\chi$ we associate the Chern connection $\nabla = \partial + \Gamma$ defined by
\begin{align*}
\Gamma^p_{jr} = \chi^{p\bar q}\partial_j\chi_{\bar qr}
\end{align*}
the torsion $T$ defined by
\begin{align*}
T^p_{jr} = \Gamma^p_{jr} - \Gamma^p_{rj}
\end{align*}
\begin{align*}
\tau_j = T^p_{jp}
\end{align*}
and the curvature $R$ defined by
\begin{align*}
R_{\bar kj}{}^p{}_r = -\partial_{\bar k}\Gamma^p_{jr}.
\end{align*}
If $\omega$ is another Hermitian metric with $d(\omega - \chi) = 0$, by a computation due to Cherrier \cite{cherrier}
\begin{align}\label{s2_LaplacianTrace}
\Delta_\omega\tr_\chi\omega
& = -\chi^{p\bar q}R_{\bar qp}(\omega) + g^{j\bar k}R_{\bar kj}{}^{p\bar q}g_{\bar qp} \\
& \ \ \ \, + g^{j\bar k}R^p{}_{p\bar kj} - g^{j\bar k}R_{\bar kj}{}^p{}_p - g^{j\bar k}\chi^{p\bar q}\chi_{\bar sr}T^r_{pj}\bar T^s_{qk} \nonumber \\
& \ \ \ \, + g^{j\bar k}\chi^{p\bar q}g_{\bar sr}\Phi^r_{pj}\bar\Phi^s_{qk} \nonumber
\end{align}
and 
\begin{align}\label{s2_LaplacianLogTrace}
\Delta_\omega\log\tr_\chi\omega
& = \frac{1}{\tr_\chi\omega}\bigg\{-\chi^{p\bar q}R_{\bar qp}(\omega) + g^{j\bar k}R_{\bar kj}{}^{p\bar q}g_{\bar qp} \\
& \ \ \ \ \ \ \ \ \ \ \ \ \, + g^{j\bar k}R^p{}_{p\bar kj} - g^{j\bar k}R_{\bar kj}{}^p{}_p - g^{j\bar k}\chi^{p\bar q}\chi_{\bar sr}T^r_{pj}\bar T^s_{qk} \nonumber \\
& \ \ \ \ \ \ \ \ \ \ \ \ \, + 2\Ree g^{j\bar k}\frac{\partial_j\tr_\chi\omega}{\tr_\chi\omega}\bar T_k + \frac{1}{\tr_\chi\omega}g^{j\bar k}T_j\bar T_k \nonumber \\
& \ \ \ \ \ \ \ \ \ \ \ \ \, + g^{j\bar k}\chi^{p\bar q}g_{\bar sr}\Phi^r_{pj}\bar\Phi^s_{qk} - \frac{1}{\tr_\chi\omega}g^{j\bar k}\Psi_j\bar\Psi_k\bigg\} \nonumber
\end{align}
where $\Phi^r_{pj} = g^{r\bar s} \nabla_pg_{\bar sj} + T^r_{pj}$ and $\Psi_j = \partial_j\tr_\chi\omega + T_j$. We have in \eqref{s2_LaplacianTrace}
\begin{align*}
g^{j\bar k}\chi^{p\bar q}g_{\bar sr}\Phi^r_{pj}\bar\Phi^s_{qk} \geq 0
\end{align*}
and by the methods used in the proof of \cite[Theorem 2.1]{tw09-2} (see also \cite[Theorem 3]{smith20}) we have in \eqref{s2_LaplacianLogTrace}
\begin{align*}
g^{j\bar k}\chi^{p\bar q}g_{\bar sr}\Phi^r_{pj}\bar\Phi^s_{qk} - \frac{1}{\tr_\chi\omega}g^{j\bar k}\Psi_j\bar\Psi_k \geq 0.
\end{align*}

\section{Hopf manifolds}\label{s3}

\noindent In this section we study equation \eqref{s1_ContinuityEquation} on round Hopf manifolds of complex dimension $n$. Specifically, let $X$ be the quotient of $\mathbb C^n\setminus\{0\}$ by the relation $(\alpha_1z_1,\dotsc,\alpha_nz_n)\sim(z_1,\dotsc,z_n)$, where $|\alpha_1|=\cdots=|\alpha_n| \neq 1$.

The Hopf metric $\omega_{\mathrm{H}}$ is defined by
\begin{align*}
g^{\mathrm{H}}_{\bar kj} = \frac{\delta_{\bar kj}}{r^2}.
\end{align*}
It follows that the Ricci curvature of $\omega_{\mathrm{H}}$ is given by
\begin{align*}
R^{\mathrm{H}}_{\bar kj} = n\frac{\delta_{\bar kj}}{r^2} - n\frac{\bar z_j z_k}{r^4}.
\end{align*}
By the Cauchy-Schwarz inequality, $\Ric(\omega_{\mathrm{H}})\geq 0$.

As in \cite{tw15}, we consider the family of Hermitian metrics $\hat\omega = \omega_{\mathrm{H}} - s\Ric(\omega_{\mathrm{H}})$, or
\begin{align*}
\hat g_{\bar kj} = (1-ns)\frac{\delta_{\bar kj}}{r^2} + ns\frac{\bar z_j z_k}{r^4}.
\end{align*}
It satisfies
\begin{align}\label{s3_OmegaHatDeterminant}
\hat\omega^n = (1-ns)^{n-1}\omega_{\mathrm{H}}^n.
\end{align}
Hence $\Ric(\hat\omega) = \Ric(\omega_{\mathrm{H}})$, and so $\hat\omega$ gives an explicit solution of equation \eqref{s1_ContinuityEquation} with $\omega_0 = \omega_{\mathrm{H}}$.

Let $\varphi_0$ be a given smooth function with $\omega_0 = \omega_{\mathrm{H}} + i\ddb\varphi_0 > 0$. By \cite[Theorem 1]{sw20}, the equation \eqref{s1_ContinuityEquation} admits a unique solution on $X\times[0,\frac{1}{n})$, which is necessarily of the form $\omega = \hat\omega + i\ddb\varphi$.

We start by proving parts \textit{(1)} and \textit{(2)} of Theorem \ref{s1_Theorem1}:

\begin{proposition}\label{s3_C0Estimate}
Suppose that $\omega = \hat\omega + i\partial\bar\partial\varphi$ solves \eqref{s1_ContinuityEquation} on $X\times[0,\frac{1}{n})$ with $\omega_0 = \omega_{\mathrm{H}} + i\partial\bar\partial\varphi_0$, subject to the normalization
\begin{align}\label{s3_Normalization}
\int_Xe^\frac{\varphi - \varphi_0}{s}\hat\omega^n = \int_X\omega^n.
\end{align}
Then $|\varphi|\leq C$ and furthermore $C^{-1}\leq\omega^n/\hat\omega^n\leq C$.
\end{proposition}
\begin{proof}
The forms of $\omega$, $\omega_0$ and $\hat\omega$ show that the equation \eqref{s1_ContinuityEquation} is
\begin{align*}
i\ddb\log\frac{\omega^n}{\hat\omega^n} = i\partial\bar\partial\frac{\varphi - \varphi_0}{s}.
\end{align*}
The normalization \eqref{s3_Normalization} implies that $\varphi$ solves the following scalar equation
\begin{align}\label{s3_ScalarEquation}
\log\frac{(\hat\omega + i\partial\bar\partial\varphi)^n}{\hat\omega^n} = \frac{\varphi-\varphi_0}{s}.
\end{align}
Now at a maximum point $x_0$ of $\varphi$ we have $i\partial\bar\partial\varphi(x_0)\leq 0$, so that $(\hat\omega(x_0) + i\partial\bar\partial\varphi(x_0))^n \leq\hat\omega(x_0)^n$. Thus by \eqref{s3_ScalarEquation}
\begin{align*}
\frac{\varphi(x_0) - \varphi_0(x_0)}{s} \leq 0
\end{align*}
or
\begin{align*}
\varphi \leq \|\varphi_0\|_{C^0}.
\end{align*}
A similar argument gives a lower bound on $\varphi$. The second estimate now follows from the first one and \eqref{s3_ScalarEquation}.
\end{proof}

\noindent Now we can prove part \textit{(3)} of Theorem \ref{s1_Theorem1}:

\begin{proposition}\label{s3_C2Estimate}
Suppose that $\omega = \hat\omega + i\partial\bar\partial\varphi$ solves \eqref{s1_ContinuityEquation} on $X\times[0,\frac{1}{n})$ with $\omega_0 = \omega_{\mathrm{H}} + i\partial\bar\partial\varphi_0$. Then $\tr_{\omega_{\mathrm{H}}}\omega\leq C$ and hence $\omega\leq C\omega_{\mathrm{H}}$.
\end{proposition}
\begin{proof}
The proof follows \cite{tw15}. Compute
\begin{align*}
g^{j\bar k}R_{\bar kj}{}^{p\bar q}g_{\bar qp} = \frac{1}{n}\tr_{\omega_{\mathrm{H}}}\omega\tr_\omega\Ric(\omega_{\mathrm{H}})
\end{align*}
and
\begin{align*}
g^{j\bar k}R^p{}_{p\bar kj} - g^{j\bar k}R_{\bar kj}{}^p{}_p - g^{j\bar k}g_{\mathrm{H}}^{p\bar q}g^{\mathrm{H}}_{\bar sr}T^r_{pj}\bar T^s_{qk} = -\frac{2}{n}\tr_\omega\Ric(\omega_{\mathrm{H}}).
\end{align*}
Therefor by \eqref{s2_LaplacianTrace}
\begin{align}\label{s3_LaplacianTrace}
\Delta_\omega\tr_{\omega_{\mathrm{H}}}\omega \geq \left(\frac{\tr_{\omega_{\mathrm{H}}}\omega}{n} - \frac{2}{n}\right)\tr_\omega\Ric(\omega_{\mathrm{H}}) - \tr_{\omega_{\mathrm{H}}}\Ric(\omega)
\end{align}
Let $x_0$ be a maximum point of $\tr_{\omega_{\mathrm{H}}}\omega$. If $\tr_{\omega_{\mathrm{H}}(x_0)}\omega(x_0) < 2$ then the result is proved, so we may assume that the quantity in parentheses in \eqref{s3_LaplacianTrace} is non-negative. But since $\Ric(\omega_{\mathrm{H}})\geq 0$, this means that the first term of \eqref{s3_LaplacianTrace} is non-negative. Thus we may use \eqref{s1_ContinuityEquation} to obtain
\begin{align*}
0
& \geq -\tr_{\omega_{\mathrm{H}}(x_0)}\Ric(\omega(x_0)) \\
& = \tr_{\omega_{\mathrm{H}}(x_0)}\frac{\omega(x_0)-\omega_0(x_0)}{s} \nonumber
\end{align*}
or
\begin{align*}
\tr_{\omega_{\mathrm{H}}}\omega \leq \|\tr_{\omega_{\mathrm{H}}}\omega_0\|_{C^0}
\end{align*}
as desired.
\end{proof}

\noindent Finally we complete the proof of Theorem \ref{s1_Theorem1}:

\begin{proposition}\label{s3_CurvatureEstimate}
Suppose that $\omega = \hat\omega + i\partial\bar\partial\varphi$ solves \eqref{s1_ContinuityEquation} on $X\times[0,\frac{1}{n})$ with $\omega_0 = \omega_{\mathrm{H}} + i\partial\bar\partial\varphi_0$. Then for $s$ sufficiently close to $1/n$, the scalar curvature of $\omega$ satisfies
\begin{align*}
\frac{C^{-1}}{(1-ns)^{1 - \frac{1}{n}}} \leq R(\omega)\leq \frac{C}{(1-ns)^{n-1}}.
\end{align*}
\end{proposition}
\begin{proof}
Equation \eqref{s1_ContinuityEquation} implies
\begin{align}
R(\omega) = \frac{\tr_\omega\omega_0 - n}{s}.
\end{align}
The above estimates together with \eqref{s4_OmegaHatDeterminant} give
\begin{align*}
\tr_\omega\omega_0
 \geq C^{-1}\left(\frac{\omega_0^n}{\omega^n}\right)^\frac{1}{n} 
 \geq C^{-1}\left(\frac{\omega_{\mathrm{H}}^n}{\hat\omega^n}\right)^\frac{1}{n} \nonumber 
 = \frac{C^{-1}}{(1-ns)^{1-\frac{1}{n}}} \nonumber
\end{align*}
and
\begin{align*}
\tr_\omega\omega_0
\leq C\frac{\omega_0^n}{\omega^n}(\tr_{\omega_0}\omega)^{n-1} 
\leq C\frac{\omega_{\mathrm{H}}^n}{\hat\omega^n} \nonumber 
= \frac{C}{(1-ns)^{n-1}}. \nonumber
\end{align*}
\end{proof}

\section{Hopf surfaces of class 1}\label{s4}

\noindent In this section we focus on Hopf surfaces of class 1. Specifically, let $X$ be the quotient of $\mathbb C^2\setminus\{0\}$ by the relation $(\alpha z,\beta w)\sim(z,w)$, where $1 < |\alpha| \leq |\beta|$. First we describe a few explicit $(1,1)$-forms on $X$ as detailed in \cite{go98} (see also \cite{edwards21}). The equation
\begin{align*}
|z|^2\Phi^{-2\gamma_1} + |w|^2\Phi^{-2\gamma_2} = 1
\end{align*}
uniquely defines a smooth function $\Phi(z,w)$ on $\mathbb C^2\setminus\{0\}$, where
\begin{align*}
\gamma_1 = \frac{\log|\alpha|}{\log|\alpha| + \log|\beta|} \ \ \ \ \ , \ \ \ \ \ \gamma_2 = \frac{\log|\beta|}{\log|\alpha| + \log|\beta|}.
\end{align*}
The Gauduchon-Ornea metric
\begin{align*}
\omega_{\mathrm{GO}} = \frac{i\ddb\Phi}{\Phi}
\end{align*}
is a well-defined Hermitian metric on $X$. Define a non-negative $(1,1)$-form
\begin{align*}
\Theta = \frac{i\partial\Phi\wedge\bar\partial\Phi}{\Phi^2}.
\end{align*}
Then we have 
\begin{align}\label{s4_ThetaLessThanOmegaGO}
\Theta\leq\omega_{\mathrm{GO}}.
\end{align}
It is also convenient to define another Hermitian metric $\chi$ by
\begin{align*}
\chi_{\bar kj} = \Phi^{-2\gamma_j}\delta_{\bar kj}
\end{align*}
as this satisfies
\begin{align}\label{s4_RicciChi}
\Ric(\chi) = 2(\omega_{\mathrm{GO}} - \Theta)
\end{align}
and so $\Ric(\chi)\geq 0$ by \eqref{s4_ThetaLessThanOmegaGO}. As in \cite{edwards21}, we consider the family of Hermitian metrics
\begin{align}\label{s4_OmegaHatDefinition}
\hat\omega & = \omega_{\mathrm{GO}} - s\Ric(\chi) \\
& = (1-2s)\omega_{\mathrm{GO}} + 2s\Theta. \nonumber
\end{align}
In contrast to the case of the round Hopf manifolds, explicit solutions to the equation \ref{s1_ContinuityEquation} are not known.
It can be shown that 
\begin{align}\label{s4_OmegaHatDeterminant}
\hat\omega^2 = (1-2s)\omega_{\mathrm{GO}}^2
\end{align}
which implies
\begin{align}\label{s4_RicciOmegaHat}
\Ric(\hat\omega) = \Ric(\omega_{\mathrm{GO}}).
\end{align}

Let $\varphi_0$ be a smooth function with $\omega_0 = \omega_{\mathrm{GO}} + i\ddb\varphi_0 > 0$. The equation \eqref{s1_ContinuityEquation} admits a unique solution on $X\times[0,\frac12)$, which necessarily has the form $\omega = \hat\omega + i\ddb\varphi$ \cite[Theorem 1]{sw20}. We begin by proving parts \textit{(1)} and \textit{(2)} of Theorem \ref{s1_Theorem2}:

\begin{proposition}\label{s4_C0Estimate}
Suppose that $\omega = \hat\omega + i\partial\bar\partial\varphi$ solves \eqref{s1_ContinuityEquation} on $X\times[0,\frac12)$ with $\omega_0 = \omega_{\mathrm{GO}} + i\partial\bar\partial\varphi_0$, subject to the normalization
\begin{align}\label{s4_Normalization}
\int_Xe^\frac{\varphi - \varphi_0}{s}\chi^2 = \int_X\omega_{\mathrm{GO}}^2.
\end{align}
Then $|\varphi|\leq C$ and furthermore $C^{-1}\leq\omega^2/\hat\omega^2\leq C$.
\end{proposition}
\begin{proof}
The forms of $\omega$, $\omega_0$, and $\hat\omega$ show that the equation \eqref{s1_ContinuityEquation} is
\begin{align*}
i\partial\bar\partial\frac{\varphi - \varphi_0}{s} = \Ric(\chi) - \Ric(\omega).
\end{align*}
Set $f = \log\frac{\omega_{\mathrm{GO}}^2}{\chi^2}$. Then by \eqref{s4_RicciOmegaHat},
\begin{align}
i\partial\bar\partial\frac{\varphi - \varphi_0}{s}
& = \Ric(\omega_{\mathrm{GO}}) - \Ric(\omega) + i\partial\bar\partial f \\
& = \Ric(\hat\omega) - \Ric(\omega) + i\partial\bar\partial f \nonumber \\
& = i\partial\bar\partial\log\frac{\omega^2}{\hat\omega^2} + i\partial\bar\partial f. \nonumber
\end{align}
The normalization \eqref{s4_Normalization} implies that $\varphi$ solves the following scalar equation
\begin{align}\label{s4_ScalarEquation}
\log\frac{(\hat\omega + i\partial\bar\partial\varphi)^2}{\hat\omega^2} = \frac{\varphi-\varphi_0}{s} - f.
\end{align}
Now at a maximum point $x_0$ of $\varphi$ we have $i\partial\bar\partial\varphi(x_0)\leq 0$, so that $(\hat\omega(x_0) + i\partial\bar\partial\varphi(x_0))^2 \leq\hat\omega(x_0)^2$. Thus by \eqref{s4_ScalarEquation}
\begin{align*}
\frac{\varphi(x_0) - \varphi_0(x_0)}{s} - f(x_0) \leq 0
\end{align*}
or
\begin{align*}
\varphi \leq \|\varphi_0\|_{C^0} + \frac12\|f\|_{C^0}.
\end{align*}
A similar argument gives a lower bound on $\varphi$. The second estimate now follows from the first one and \eqref{s4_ScalarEquation}.
\end{proof}

\noindent From here, we can prove part \textit{(3)} of Theorem \ref{s1_Theorem2}:

\begin{proposition}\label{s4_C2Estimate}
Suppose that $\omega = \hat\omega + i\partial\bar\partial\varphi$ solves \eqref{s1_ContinuityEquation} on $X\times[0,\frac12)$ with $\omega_0 = \omega_{\mathrm{GO}} + i\partial\bar\partial\varphi_0$. Then $\tr_{\omega_{\mathrm{GO}}}\omega\leq C$ and hence $\omega\leq C\omega_{\mathrm{GO}}$.
\end{proposition}
\begin{proof}
The proof follows \cite{edwards21}. We will obtain a bound on $\tr_\chi\omega$, since this is equivalent to the desired bound. Compute
\begin{align*}
g^{j\bar k}R_{\bar kj}{}^{p\bar q}g_{\bar qp}
& = \gamma_p\chi^{p\bar q}g_{\bar qp}\tr_\omega\Ric(\chi) \\
& \geq C^{-1}\tr_\chi\omega\tr_\omega\Ric(\chi) \nonumber
\end{align*}
by taking $C^{-1}\leq\min\gamma_j$,
and
\begin{align*}
g^{j\bar k}R^p{}_{p\bar kj} - g^{j\bar k}R_{\bar kj}{}^p{}_p - g^{j\bar k}\chi^{p\bar q}\chi_{\bar sr}T^r_{pj}\bar T^s_{qk} \geq -C\tr_\omega\omega_{\mathrm{GO}}.
\end{align*}
Therefor by \eqref{s2_LaplacianTrace}
\begin{align*}
\Delta_\omega\tr_\chi\omega \geq C^{-1}\tr_\chi\omega\tr_\omega\Ric(\chi) - C\tr_\omega\omega_{\mathrm{GO}} - \tr_\chi\Ric(\omega)
\end{align*}
and since
\begin{align*}
\Delta_\omega\varphi
& = 2 - \tr_\omega\omega_{\mathrm{GO}} + s\tr_\omega\Ric(\chi) \\
& \leq 2 - \tr_\omega\omega_{\mathrm{GO}} + \frac12\tr_\omega\Ric(\chi) \nonumber
\end{align*}
also
\begin{align}\label{s4_LaplacianTrace}
\Delta_\omega(\tr_\chi\omega & - A\varphi) \\
& \geq \left(\frac{\tr_\chi\omega}{C} - \frac{A}{2}\right)\tr_\omega\Ric(\chi) + (A - C)\tr_\omega\omega_{\mathrm{GO}} - \tr_\chi\Ric(\omega) - 2A \nonumber \\
& \geq \left(\frac{\tr_\chi\omega}{C} - \frac{A}{2}\right)\tr_\omega\Ric(\chi) - \tr_\chi\Ric(\omega) - 2A \nonumber
\end{align}
after choosing $A$ larger than $C$.

Let $x_0$ be a maximum point of $\tr_\chi\omega-A\varphi$. If
\begin{align*}
\frac{\tr_{\chi(x_0)}\omega(x_0)}{C} - \frac{A}{2} < 0
\end{align*}
then the result is proved, so we may assume that the coefficient of $\tr_\omega\Ric(\chi)$ in \eqref{s4_LaplacianTrace} is non-negative. But since $\Ric(\chi)\geq 0$, this means that the first term of \eqref{s4_LaplacianTrace} is non-negative. Thus we may use \eqref{s1_ContinuityEquation} to obtain
\begin{align*}
0
& \geq -\tr_{\chi(x_0)}\Ric(\omega(x_0)) - 2A \\
& = \tr_{\chi(x_0)}\frac{\omega(x_0)-\omega_0(x_0)}{s} - 2A \nonumber
\end{align*}
or
\begin{align*}
\tr_{\chi(x_0)}\omega(x_0) \leq \tr_{\chi(x_0)}\omega_0(x_0) + A
\end{align*}
which implies
\begin{align*}
\tr_\chi\omega \leq \|\tr_\chi\omega_0\|_{C^0} + 2A\|\varphi\|_{C^0} + A.
\end{align*}
Finally, applying Proposition \ref{s4_C0Estimate} proves the result.
\end{proof}

\noindent We are now ready to finish the proof of Theorem \ref{s1_Theorem2}:

\begin{proposition}\label{s4_CurvatureEstimate}
Suppose that $\omega = \hat\omega + i\partial\bar\partial\varphi$ solves \eqref{s1_ContinuityEquation} on $X\times[0,\frac12)$ with $\omega_0 = \omega_{\mathrm{GO}} + i\partial\bar\partial\varphi_0$. Then $\tr_{\omega_{\mathrm{GO}}}\omega\leq C$ and hence $\omega\leq C\omega_{\mathrm{GO}}$. Then for $s$ sufficiently close to $1/2$, the scalar curvature of $\omega$ satisfies
\begin{align*}
\frac{C^{-1}}{\sqrt{1-2s}} \leq R(\omega)\leq \frac{C}{1-2s}.
\end{align*}
\end{proposition}
\begin{proof}
Equation \eqref{s1_ContinuityEquation} implies
\begin{align*}
R(\omega) = \frac{\tr_\omega\omega_0 - 2}{s}.
\end{align*}
The above estimates together with \eqref{s3_OmegaHatDeterminant} give
\begin{align*}
\tr_\omega\omega_0
 \geq C^{-1}\left(\frac{\omega_0^2}{\omega^2}\right)^\frac{1}{2} 
 \geq C^{-1}\left(\frac{\omega_{\mathrm{GO}}^2}{\hat\omega^2}\right)^\frac{1}{2} \nonumber 
 = \frac{C^{-1}}{\sqrt{1-2s}} \nonumber
\end{align*}
and
\begin{align*}
\tr_\omega\omega_0
 = \frac{\omega_0^2}{\omega^2}\tr_{\omega_0}\omega 
 \leq C\frac{\omega_{\mathrm{GO}}^2}{\hat\omega^2} \nonumber 
 = \frac{C}{1-2s}. \nonumber
\end{align*}
\end{proof}

\section{Inoue surfaces}\label{s5}

\noindent Let $X$ be the quotient of $\mathbb C \times \mathbb H$ by the relations $(\mu z,\lambda w) \sim (z,w)$ and $(z + m_j, w + \ell_j) \sim (z,w)$ where $M$ is a given matrix in $SL_3(\mathbb Z)$ and $Mm = \mu m$, $M\ell = \lambda\ell$. On $\mathbb C \times \mathbb H$ we use coordinates $(z,w)$ and set $y = \Imm w$. The expressions
\begin{align*}
\alpha = \frac{1}{4y^2}idw\wedge d\bar w
\end{align*}
\begin{align*}
\beta = y\,idz\wedge d\bar z
\end{align*}
define $(1,1)$-forms on $X$. Let $\Omega = \alpha\wedge\beta$. Then
\begin{align}\label{s5_iddblogOmega}
i\ddb\log\Omega = \alpha.
\end{align}
\noindent By the recent work of Angella-Tosatti \cite{ta21}, in the $i\ddb$-class of any Gauduchon metric, there is a metric $\omega_{\mathrm{LF}}$ satisfying
\begin{align}\label{s5_OmegaLF}
c\alpha\wedge\omega_{\mathrm{LF}} = \Omega.
\end{align}
Using this, we may assume for the proof of Theorem \ref{s1_Theorem3} that $\omega_0 = \omega_{\mathrm{LF}} + i\ddb\varphi_0$. We consider the family of Hermitian metrics
\begin{align*}
\hat\omega
& = (\omega_{\mathrm{LF}} + s\alpha)/(s+1).
\end{align*}

For example, the Tricerri metric $\omega_{\mathrm{Tr}} = 4\alpha + \beta$ satisfies \eqref{s5_OmegaLF}. It also satisfies $\Ric(\omega_{\mathrm{Tr}}) = -\alpha$. In this case $\hat\omega = (\omega_{\mathrm{Tr}} - \Ric(\omega_{\mathrm{Tr}}))/(s+1)$; the reader can check that $\Ric(\hat\omega) = \Ric(\omega_{\mathrm{Tr}})$, and so $\hat\omega$ gives an explicit solution of equation \eqref{s1_NormalizedContinuityEquation} with $\omega_0 = \omega_{\mathrm{Tr}}$.

Let $\varphi_0$ be a given smooth function with $\omega_0 = \omega_{\mathrm{LF}} + i\ddb\varphi_0 > 0$. By \cite[Theorem 1]{sw20}, the equation \eqref{s1_NormalizedContinuityEquation} admits a unique solution on $X\times[0,\infty)$ of the form $\omega = \hat\omega + i\ddb\varphi$.

Here we prove parts \textit{(1)} and \textit{(2)} of Theorem \ref{s1_Theorem3}:

\begin{proposition}\label{s5_C0Estimate}
Suppose that $\omega = \hat\omega + i\partial\bar\partial\varphi$ solves \eqref{s1_NormalizedContinuityEquation} on $X\times[0,\infty)$ with $\omega_0 = \omega_{\mathrm{LF}} + i\partial\bar\partial\varphi_0$, subject to the normalization
\begin{align}\label{s5_Normalization}
\int_Xe^\frac{(s+1)\varphi - \varphi_0}{s}\Omega = \frac{(s+1)^2}{2s}\int_X\omega^2.
\end{align}
Then $|\varphi|\leq C/(s+1)$. Furthermore
\begin{align}\label{s5_PreciseDeterminantEstimate}
-\frac{C}{s}\leq\log\frac{\omega^2}{\hat\omega^2}\leq \frac{C}{s}
\end{align}
and in particular
\begin{align}\label{s5_DeterminantEstimate}
C^{-1}\leq\frac{\omega^2}{\hat\omega^2}\leq C.
\end{align}
\end{proposition}
\begin{proof}
The forms of $\omega$, $\omega_0$, and $\hat\omega$ together with \eqref{s5_iddblogOmega}, show that the equation \eqref{s1_NormalizedContinuityEquation} is
\begin{align*}
i\ddb\log\frac{\omega^2}{\Omega} = i\partial\bar\partial\frac{(s+1)\varphi - \varphi_0}{s}.
\end{align*}
The normalization \eqref{s5_Normalization} implies that $\varphi$ solves the following scalar equation
\begin{align}\label{s5_ScalarEquation}
\log\frac{(s+1)^2(\hat\omega + i\partial\bar\partial\varphi)^2}{2s\Omega} = \frac{(s+1)\varphi-\varphi_0}{s}.
\end{align}
Now at a maximum point $x_0$ of $\varphi$ we have $i\partial\bar\partial\varphi(x_0)\leq 0$, so that $(\hat\omega(x_0) + i\partial\bar\partial\varphi(x_0))^2 \leq\hat\omega(x_0)^2$. Since $(s+1)^2\hat\omega^2 = \omega_{\mathrm{LF}}^2 + 2s\Omega$, \eqref{s5_ScalarEquation} becomes
\begin{align*}
\frac{(s+1)\varphi(x_0) - \varphi_0(x_0)}{s} \leq \log\left(1 + \frac{f(x_0)}{s}\right)
\end{align*}
where $f = \omega_{\mathrm{LF}}^2/2\Omega$. This implies
\begin{align*}
(s+1)\varphi(x_0) - \varphi_0(x_0)
 \leq s\log\left(1 + \frac{f(x_0)}{s}\right)  \leq f(x_0) \nonumber
\end{align*}
or
\begin{align*}
\varphi \leq \frac{\|\varphi_0\|_{C^0} + \|f\|_{C^0}}{s+1}.
\end{align*}
A similar argument gives
\begin{align*}
\varphi \geq -\frac{\|\varphi_0\|_{C^0}}{s+1}.
\end{align*}
Now \eqref{s5_ScalarEquation} gives
\begin{align*}
-\frac{C}{s}\leq\log\frac{(s+1)^2\omega^2}{2s\Omega}\leq \frac{C}{s}.
\end{align*}
Since $(s+1)^2\hat\omega^2 = \omega_{\mathrm{LF}}^2 + 2s\Omega$ we have $2s\Omega \leq (s+1)^2\hat\omega^2 \leq (2s+C)\Omega$ for $s\geq 1$. So
\begin{align*}
\log\frac{\omega^2}{\hat\omega^2}
 = \log\frac{(s+1)^2\omega^2}{(s+1)^2\hat\omega^2}  \leq \log\frac{(s+1)^2\omega^2}{2s\Omega} \nonumber  \leq \frac{C}{s} \nonumber
\end{align*}
and
\begin{align*}
\log\frac{\omega^2}{\hat\omega^2}
& = \log\frac{(s+1)^2\omega^2}{(s+1)^2\hat\omega^2} \\
& \geq \log\frac{(s+1)^2\omega^2}{(2s+C)\Omega} \nonumber \\
& = \log\frac{(s+1)^2\omega^2}{2s\Omega} - \log\left(1+\frac{C}{2s}\right) \nonumber \\
& \geq -\frac{C}{s} \nonumber
\end{align*}
since $\log(1 + C/2s) \leq C/2s$. This proves \eqref{s5_PreciseDeterminantEstimate}.
\end{proof}

\noindent Next we prove a first step towards part \textit{(3)} of Theorem \ref{s1_Theorem3}:

\begin{proposition}\label{s5_C2Estimate}
Suppose that $\omega = \hat\omega + i\partial\bar\partial\varphi$ solves \eqref{s1_NormalizedContinuityEquation} on $X\times[0,\infty)$ with $\omega_0 = \omega_{\mathrm{LF}} + i\partial\bar\partial\varphi_0$. Then $\tr_{\hat\omega}\omega\leq C$ and $\tr_\omega\hat\omega \leq C$, and hence $C^{-1}\hat\omega\leq\omega\leq C\hat\omega$.
\end{proposition}
\begin{proof}
The equation \eqref{s1_NormalizedContinuityEquation} gives
\begin{align}\label{s5_C2EstimateComputation1}
-\tr_{\hat\omega}\Ric(\omega)
& = \frac{(s+1)\tr_{\hat\omega}\omega - \tr_{\hat\omega}\omega_0}{s} \\
& \geq -\frac{1}{s}\tr_{\hat\omega}\omega_0 \nonumber \\
& \geq -C \nonumber
\end{align}
for $s\geq 1$, since $\tr_{\hat\omega}\omega_0 \leq (s+1)\tr_{\omega_{\mathrm{LF}}}\omega_0 \leq C(s+1)$. Computing as in \cite[Lemma 4.1]{twy15}, we have
\begin{align}\label{s5_C2EstimateComputation2}
g^{j\bar k}R_{\bar kj}{}^{p\bar q}g_{\bar qp} \geq -C\sqrt{s+1}\tr_{\hat\omega}\omega\tr_\omega\hat\omega
\end{align}
and
\begin{align}\label{s5_C2EstimateComputation3}
g^{j\bar k}R^p{}_{p\bar kj} - g^{j\bar k}R_{\bar kj}{}^p{}_p - g^{j\bar k}\hat g^{p\bar q}\hat g_{\bar sr}T^r_{pj}\bar T^s_{qk} \geq -C\tr_\omega\hat\omega.
\end{align}
Substituting \eqref{s5_C2EstimateComputation1}, \eqref{s5_C2EstimateComputation2}, \eqref{s5_C2EstimateComputation3} into \eqref{s2_LaplacianLogTrace} yields
\begin{align}\label{s5_LaplacianLogTrace}
\Delta_\omega\log\tr_{\hat\omega}\omega
& \geq \frac{1}{\tr_{\hat\omega}\omega}\bigg\{-C\sqrt{s+1}\tr_{\hat\omega}\omega\tr_\omega\hat\omega - C\tr_\omega\hat\omega - C + 2\Ree g^{j\bar k}\frac{\partial_j\tr_{\hat\omega}\omega}{\tr_{\hat\omega}\omega}\bar T_k\bigg\} \\
& \geq -C\sqrt{s+1}\tr_\omega\hat\omega - C + \frac{1}{\tr_{\hat\omega}\omega}2\Ree g^{j\bar k}\frac{\partial_j\tr_{\hat\omega}\omega}{\tr_{\hat\omega}\omega}\bar T_k. \nonumber
\end{align}
Here we have used $\tr_{\hat\omega}\omega \geq C^{-1}$, which follows from the determinant estimate \eqref{s5_DeterminantEstimate}.
We have
\begin{align}\label{s5_LaplacianPhi}
\Delta_\omega\varphi = 2 - \tr_\omega\hat\omega.
\end{align}
Let $\psi = \sqrt{s+1}\varphi + C$; by Proposition \ref{s5_C0Estimate} we may assume that $\psi \geq 1$. Then $0 < 1/\psi \leq 1$ and
\begin{align}\label{s5_LaplacianPsi}
\Delta_\omega\frac{1}{\psi}
& = \frac{2(s+1)}{\psi^3}|\partial\varphi|_\omega^2 - \frac{2\sqrt{s+1}}{\psi^2} + \frac{\sqrt{s+1}}{\psi^2}\tr_\omega\hat\omega \\
& \geq \frac{2(s+1)}{\psi^3}|\partial\varphi|_\omega^2 - 2\sqrt{s+1}. \nonumber
\end{align}
Combining \eqref{s5_LaplacianLogTrace}, \eqref{s5_LaplacianPhi}, \eqref{s5_LaplacianPsi} gives
\begin{align}\label{s5_LaplacianQ1}
\Delta_\omega\left(\log\tr_{\hat\omega}\omega - A\sqrt{s+1}\varphi + \frac{1}{\psi}\right) & \geq \frac{1}{\tr_{\hat\omega}\omega}2\Ree g^{j\bar k}\frac{\partial_j\tr_{\hat\omega}\omega}{\tr_{\hat\omega}\omega}\bar T_k + (A - C)\sqrt{s+1}\tr_\omega\hat\omega \nonumber \\
& \ \ \ \ \ \ \ \ \ \ \ \ \ \ \ + \frac{2(s+1)}{\psi^3}|\partial\varphi|_\omega^2 - CA\sqrt{s+1}
\end{align}
as long as $A$ is not too small.

We now work at a maximum point $x_0$ of $\log\tr_{\hat\omega}\omega - A\sqrt{s+1}\varphi + 1/\psi$, where
\begin{align}\label{s5_LaplacianQ2}
0 & \geq \frac{1}{\tr_{\hat\omega}\omega}2\Ree g^{j\bar k}\frac{\partial_j\tr_{\hat\omega}\omega}{\tr_{\hat\omega}\omega}\bar T_k + (A - C)\sqrt{s+1}\tr_\omega\hat\omega \\
& \ \ \ \ \ \ \ \ \ \ \ \ \ \ \ + \frac{2(s+1)}{\psi^3}|\partial\varphi|_\omega^2 - CA\sqrt{s+1} \nonumber
\end{align}
and
\begin{align*}
\frac{\partial_j\tr_{\hat\omega}\omega}{\tr_{\hat\omega}\omega} = A\sqrt{s+1}\partial_j\varphi + \frac{\sqrt{s+1}\partial_j\varphi}{\psi^2}.
\end{align*}
The latter implies
\begin{align}\label{s5_CrossTerm}
\frac{1}{\tr_{\hat\omega}\omega}2\Ree g^{j\bar k}\frac{\partial_j\tr_{\hat\omega}\omega}{\tr_{\hat\omega}\omega}\bar T_k
& = \frac{1}{\tr_{\hat\omega}\omega}2\Ree g^{j\bar k}A\sqrt{s+1}\partial_j\varphi\,\bar T_k + \frac{1}{\tr_{\hat\omega}\omega}2\Ree g^{j\bar k}\frac{\sqrt{s+1}\partial_j\varphi}{\psi^2}\bar T_k \nonumber \\
& \geq -\frac{2(s+1)}{\psi^3}|\partial\varphi|_\omega^2 - \frac{A^2\psi^3|\tau|_\omega^2}{(\tr_{\hat\omega}\omega)^2} - \frac{|\tau|_\omega^2}{\psi(\tr_{\hat\omega}\omega)^2} \nonumber \\
& \geq -\frac{2(s+1)}{\psi^3}|\partial\varphi|_\omega^2 - \frac{CA^2\psi^3}{(\tr_{\hat\omega}\omega)^2}\tr_\omega\hat\omega - C\tr_\omega\hat\omega
\end{align}
where here we have used $|\tau|_\omega^2 \leq C\tr_\omega\hat\omega$ as well as $\psi(\tr_{\hat\omega}\omega)^2 \geq C^{-1}$. Substituting \eqref{s5_CrossTerm} into \eqref{s5_LaplacianQ2} gives
\begin{align}\label{s5_LaplacianQ3}
0
& \geq \left((A - C)\sqrt{s+1} - C - \frac{CA^2\psi^3}{(\tr_{\hat\omega}\omega)^2}\right)\tr_\omega\hat\omega - CA\sqrt{s+1} \\
& \geq \left(\sqrt{s+1} + 1 - \frac{C\psi^3}{(\tr_{\hat\omega}\omega)^2}\right)\tr_\omega\hat\omega - C\sqrt{s+1} \nonumber
\end{align}
after choosing $A$ so that $(A-C)\sqrt{s+1} - C \geq \sqrt{s+1} + 1$. Now we must have either
\begin{align}\label{s5_Case1}
\frac{C\psi^3}{(\tr_{\hat\omega}\omega)^2} > 1
\end{align}
or
\begin{align}\label{s5_Case2}
0 \geq \tr_\omega\hat\omega - C.
\end{align}
The $C^0$ estimate \eqref{s5_C0Estimate} applied to \eqref{s5_Case1} and the determinant estimate \eqref{s5_DeterminantEstimate} applied to \eqref{s5_Case2} show that in both cases we have
\begin{align*}
\tr_{\hat\omega(x_0)}\omega(x_0) \leq C.
\end{align*}
Finally applying once more Proposition \eqref{s5_C0Estimate} while recalling $0 < 1/\psi \leq 1$ gives $\tr_{\hat\omega}\omega \leq C$. The estimate $\tr_\omega\hat\omega \leq C$ now follows from \eqref{s5_DeterminantEstimate}.
\end{proof}

\noindent We now refine Proposition \ref{s5_C2Estimate} as follows, to obtain part \textit{(3)} of Theorem \ref{s1_Theorem3}:

\begin{proposition}\label{s5_C2DecayEstimate}
Suppose that $\omega = \hat\omega + i\partial\bar\partial\varphi$ solves \eqref{s1_NormalizedContinuityEquation} on $X\times[0,\infty)$ with $\omega_0 = \omega_{\mathrm{LF}} + i\partial\bar\partial\varphi_0$. Then $\tr_{\hat\omega}\omega - 2\leq C(s+1)^{-\frac{1}{4}}$ and $\tr_\omega\hat\omega - 2 \leq C(s+1)^{-\frac{1}{4}}$, and furthermore
\begin{align}\label{s5_ImprovedUniformEquivalence}
(1-C(s+1)^{-1/8})\hat\omega\leq\omega\leq (1+C(s+1)^{-1/8})\hat\omega.
\end{align}
\end{proposition}
\begin{proof}
Substituting the estimates $\tr_{\hat\omega}\omega \leq C$ and $\tr_\omega\hat\omega \leq C$ into our earlier computations \eqref{s5_C2EstimateComputation1}, \eqref{s5_C2EstimateComputation2}, \eqref{s5_C2EstimateComputation3} we obtain from \eqref{s2_LaplacianTrace}
\begin{align}
\Delta_\omega\tr_{\hat\omega}\omega \geq -C\sqrt{s+1}.
\end{align}
Thus
\begin{align}\label{s5_RefinedC2Estimate1}
\Delta_\omega(\tr_{\hat\omega}\omega - A\varphi) \geq A(\tr_\omega\hat\omega - 2) - C\sqrt{s+1}.
\end{align}
Using
\begin{align*}
\tr_\omega\hat\omega
& = \frac{\hat\omega^2}{\omega^2}\tr_{\hat\omega}\omega \\
& = \tr_{\hat\omega}\omega - \left(1 - \frac{\hat\omega^2}{\omega^2}\right)\tr_{\hat\omega}\omega \nonumber
\end{align*}
\eqref{s5_RefinedC2Estimate1} becomes
\begin{align}\label{s5_RefinedC2Estimate2}
\Delta_\omega(\tr_{\hat\omega}\omega - A\varphi) \geq A(\tr_{\hat\omega}\omega - 2) - A\left(1 - \frac{\hat\omega^2}{\omega^2}\right)\tr_{\hat\omega}\omega - C\sqrt{s+1}.
\end{align}
By \eqref{s5_PreciseDeterminantEstimate}
\begin{align}\label{s5_VolumeDecayEstimate}
1 - \frac{\omega^2}{\hat\omega^2}
& \leq 1 - e^{-C/s} \\
& \leq C(s+1)^{-\frac{1}{4}} \nonumber
\end{align}
for $s\geq 1$. Now \eqref{s5_VolumeDecayEstimate} and Proposition \ref{s5_C2Estimate} imply
\begin{align*}
\Delta_\omega(\tr_{\hat\omega}\omega - A\varphi) \geq A(\tr_{\hat\omega}\omega - 2) - CA(s + 1)^{-\frac{1}{4}} - C\sqrt{s+1}.
\end{align*}
At a maximum point $x_0$ of $\tr_{\hat\omega}\omega - A\varphi$,
\begin{align*}
0 \geq A(\tr_{\hat\omega(x_0)}\omega(x_0) - 2) - CA(s + 1)^{-\frac{1}{4}} - C\sqrt{s+1}
\end{align*}
or
\begin{align*}
\tr_{\hat\omega(x_0)}\omega(x_0) - 2 \leq \frac{C\sqrt{s+1}}{A} + C(s+1)^{-\frac{1}{4}}.
\end{align*}
Hence
\begin{align*}
\tr_{\hat\omega}\omega - 2 \leq A(\varphi - \inf_X\varphi) + \frac{C\sqrt{s+1}}{A} + C(s+1)^{-\frac{1}{4}}
\end{align*}
and by Proposition \ref{s5_C0Estimate}
\begin{align*}
\tr_{\hat\omega}\omega - 2 \leq \frac{CA}{s+1} + \frac{C\sqrt{s+1}}{A} + C(s+1)^{-\frac{1}{4}}.
\end{align*}
Choose $A = (s+1)^\frac{3}{4}$; then
\begin{align}\label{s5_ImprovedTraceEstimate}
\tr_{\hat\omega}\omega - 2 \leq C(s+1)^{-\frac{1}{4}}.
\end{align}
Now by \eqref{s5_ImprovedTraceEstimate} and \eqref{s5_PreciseDeterminantEstimate}
\begin{align*}
\tr_\omega\hat\omega - 2
& = \frac{\hat\omega^2}{\omega^2}\tr_{\hat\omega}\omega - 2 \\
& = \frac{\hat\omega^2}{\omega^2}\left((\tr_{\hat\omega}\omega - 2) + 2\left(1 - \frac{\omega^2}{\hat\omega^2}\right)\right) \nonumber \\
& \leq Ce^{C/s}(s+1)^{-\frac{1}{4}} \nonumber \\
& \leq C(s+1)^{-\frac{1}{4}}. \nonumber
\end{align*}
for $s \geq 1$.

By an elementary inequality \cite[Lemma 7.4]{twy15}, the inequalities $\lambda_1 + \lambda_2 \leq 2 + \epsilon$ and $\frac{1}{\lambda_1} + \frac{1}{\lambda_2} \leq 2 + \epsilon$ imply $|\lambda_j - 1| \leq 2\sqrt{\epsilon}$, so the estimates $\tr_{\hat\omega}\omega - 2\leq C(s+1)^{-\frac{1}{4}}$ and $\tr_\omega\hat\omega - 2\leq C(s+1)^{-\frac{1}{4}}$ give \eqref{s5_ImprovedUniformEquivalence}.
\end{proof}

\noindent Finally we obtain part \textit{(4)} of Theorem \ref{s1_Theorem3}:

\begin{proposition}
Suppose that $\omega = \hat\omega + i\partial\bar\partial\varphi$ solves \eqref{s1_NormalizedContinuityEquation} on $X\times[0,\infty)$ with $\omega_0 = \omega_{\mathrm{LF}} + i\partial\bar\partial\varphi_0$. Then
\begin{align}\label{s5_RicciEstimate}
-C\omega\leq\Ric(\omega)\leq C\omega.
\end{align}
\end{proposition}
\begin{proof}
By \eqref{s1_NormalizedContinuityEquation}
\begin{align*}
\Ric(\omega) = \frac{\omega_0 - (s+1)\omega}{s}.
\end{align*}
This implies
\begin{align*}
-\left(1 + \frac{1}{s}\right)\omega \leq \Ric(\omega) \leq \frac{1}{s}\omega_0.
\end{align*}
Now $\omega_0 \leq C\omega_{\mathrm{LF}} \leq C(s+1)\hat\omega \leq C(s+1)\omega$ by Proposition \ref{s5_C2Estimate}, so in fact
\begin{align*}
-\left(1 + \frac{1}{s}\right)\omega \leq \Ric(\omega) \leq C\left(1 + \frac{1}{s}\right)\omega
\end{align*}
which implies \eqref{s5_RicciEstimate}.
\end{proof}

\bibliography{bib_for_all}
\end{document}